\documentclass[10pt,a4paper]{article}
\usepackage{amssymb,amsthm,amsmath,latexsym}
\usepackage[varg]{pxfonts}

\newtheorem{theorem}{Theorem}[section]

\newtheorem{corollary}[theorem]{Corollary}
\newtheorem{definition}[theorem]{Definition}
\newtheorem{example}[theorem]{Example}
\newtheorem{lemma}[theorem]{Lemma}

\newtheorem{remark}[theorem]{Remark}
\newtheorem{question}[theorem]{Question}
\newtheorem{nota}[theorem]{}

\def\Lip{\mathrm{Lip}}

\def\to{\rightarrow}

\def\R{\mathbb R}
\def\N{\mathbb N}

\def\e{\varepsilon}

\begin{document}

\title{When is $U(X)$ a ring?}

\author{Javier Cabello S\'anchez}
% \ead{coco@unex.es}
% \address[affil1]{Departamento de Matem\'aticas de la Universidad de Extremadura}
\newcommand{\AuthorNames}{J. Cabello S\'anchez}

\newcommand{\FilMSC}{Primary 54C10; Secondary 54C30, 54E40, 54D20} 
\newcommand{\FilKeywords}{Rings of uniformly continuous functions, 
Bourbaki-boundedness, uniform isolation. }
 \newcommand{\FilSupport}
{The author was partially supported by MINECO (Ministerio de Econom\'ia y 
Competitividad) projects MTM2010-20190-C02-01 and MTM2013-45643-C02-01-P 
and by the Junta de Extremadura project GR10113. }

\maketitle

\begin{abstract}
In this short paper, we will show that the space of real valued uniformly 
continuous functions defined on a metric space $(X,d)$ is a ring if and only 
if every subset $A\subset X$ has one of the following properties: 

\begin{itemize}
\item $A$ is Bourbaki-bounded, i.e., every uniformly continuous function 
on $X$ is bounded on $A$. 
\item $A$ contains an infinite uniformly isolated subset, i.e., there exist $\delta>0$ 
and an infinite subset $F\subset A$ such that $d(a,x)\geq \delta$ for every 
$a\in F, x\in X$. 
\end{itemize}
\end{abstract}

\makeatletter
\renewcommand\@makefnmark
{\mbox{\textsuperscript{\normalfont\@thefnmark)}}}
\makeatother

\section*{Introduction}

Even when it is quite surprising, it seems that there is no characterization of 
metric spaces whom real valued uniformly continuous functions have ring structure. 
Everybody knows that $C(X)$ is a ring whenever $X$ is a topological space, as well 
as $\Lip(X)$ is a ring if and only if $X$ has finite diametre. The main result in 
this paper solves this lack of information about $U(X)$, but with a somehow disgusting 
statement. Recall the Atsuji characterization of {\em UC} spaces, where 
$X'$ is the set of accumulation points in $X$ (\cite{A}): 

\vspace{3 mm}

\noindent{\bf Theorem }(\cite{A}){\bf .}
{\em Let $X$ be a metric space. Then, $U(X)=C(X)$ if and only if $X'$ is compact 
and every closed $I\subset X\setminus X'$ is uniformly isolated.} 

\vspace{3 mm}

An alternative way to state this result is: 

\vspace{3 mm}

\noindent{\bf Theorem.}
{\em Let $X$ be a metric space. Then, $U(X)=C(X)$ if and only if $X'$ is compact and, 
for every $\delta>0$, the set $I=\{x\in X:d(x,X')\geq \delta\}$ is uniformly isolated.}

\vspace{3 mm}

Our ideal statement would be entirely analogous to the second one, writing 
``$X'$ is Bourbaki-bounded" instead of ``$X'$ is compact", but we cannot 
ensure that $X\setminus X'$ will have this property. Namely, we have found two 
main problems. The first one is that there may be some sequences near $X'$ formed 
by isolated points which do not affect to its Bourbaki-boundedness --we will 
put an example (see \ref{ejBb}). This could be solved by changing $X'$ by another 
Bourbaki-bounded $A$, and stating it as follows: 

\vspace{3 mm}

\noindent{\bf Conjecture.}
{\em Let $X$ be a metric space. Then, $U(X)$ is a ring if and only if there exists a
Bourbaki-bounded subset $A$ such that 
for every $\delta>0$, the set $\{x\in X:d(x,A)\geq \delta\}$ is uniformly isolated.}

\vspace{3 mm}

The second problem is that we have not been able to show that this is true, 
we have just shown that if there exists such an $A$, then $U(X)$ is a ring 
--even this seems to be new since, in the recent paper \cite{N}, it is
shown that $U(X)$ is a ring whenever 
$X=F\cup I$, where $F$ is Bourbaki-bounded and $I$ uniformly isolated. 

The main problem in this paper has been widely studied (see, e.g., \cite{A}, 
\cite{N}, \cite{NZ}) along with problems about characterizing the couples 
of uniformly continuous functions whom product remains uniformly continuous 
(again \cite{A}, but also \cite{BN}, \cite{C}). 

\section{Bourbaki-bounded (sub)spaces}

We are about to explain some properties of Bourbaki-bounded spaces before coming up to 
the main result. 

\begin{nota}
Given a metric space $X$, we will denote by $U(X)$ the set of real valued, 
uniformly continuous functions defined on $X$. The subset of bounded 
functions in $U(X)$ will be denoted by $U^*(X)$. 
Recall that $U^*(X)$ is always a ring. 
\end{nota}

\begin{definition} \rm
A metric space $X$ is said to be {\em Bourbaki-bounded} if $U(X)=U^*(X)$. 
$A\subset X$ is a Bourbaki-bounded subset if the map 
$f_{|A}:A\to\R, f_{|A}(x)=f(x)$ is bounded for every $f\in U(X)$. 
\end{definition}

\begin{nota}
Given $A\subset X$ and $\gamma>0$, we will denote by $A^\gamma$ the set 
of points whom distance to $A$ is not greater than $\gamma$: 
$A^{1, \gamma}= A^\gamma=\{x\in X:d(x, A)\leq\gamma\}$. Inductively, 
$A^{k, \gamma}=\{x\in X:d(x, A^{k-1, \gamma})\leq\gamma\}$
\end{nota}

\begin{nota}
$\{U_j:j\in J\}$ is a {\em uniform} covering of $X$ whenever there 
exists $\e>0$ such that, for every $x\in X$ there is $j\in J$ such that 
$B(x,\e)\subset U_j$. 
\end{nota}

\begin{nota}
A covering $\{U_j:j\in J\}$ is said to be {\em star-finite} if 
$\{i\in J:U_i\cap U_j\neq\emptyset\}$ for every $j\in J$. 
\end{nota}

Let us recall some characterizations of Bourbaki-bounded spaces:

\begin{theorem}
Let $(X,d)$ be a metric space. Then, the following statements are equivalent: 
\begin{enumerate}
\item $X$ is a Bourbaki-bounded metric space. 
\item For every metric space $Y$ and every uniformly continuous function 
$f:X\to Y$, $f(X)$ is bounded in $Y$ (\cite{A}). 
\item $X$ is $d'$-bounded for every metric $d'$ uniformly equivalent to $d$ (\cite{H}). 
\item Every star-finite uniform cover of $X$ is finite (\cite{Nj}). 
\item Every countable $B\subset X$ is a Bourbaki-bounded subset in $X$ (\cite{GM}). 
\item For every $\gamma>0$, there exist $M_1, M_2\in\N, x_1, x_2,\ldots, 
x_{M_1}\in X$ such that $X=\{x_1, x_2,\ldots, x_{M_1}\}^{M_2, \gamma}$ (\cite{A}). 
\end{enumerate}
\end{theorem}

When $X$ fulfills any of the previous conditions, it is also said that 
$X$ is finitely chainable, because of the last condition. 

The following translations to Bourbaki-bounded subsets are easy to check: 

\begin{theorem}
Let $A\subset X$. Then, the following statements are equivalent:
\begin{enumerate}
\item $A$ is a Bourbaki-bounded subset. 
\item For every metric space $Y$ and every uniformly continuous function
$f:X\to Y$, $f(A)$ is bounded in $Y$. 
\item $A$ is $d'$-bounded for every metric $d'$ uniformly equivalent to $d$ on $X$. 
\item Every countable $B\subset A$ is a Bourbaki-bounded subset in $X$.
\item For every $\gamma>0$, there exist $M_1, M_2\in\N, x_1, x_2,\ldots, 
x_{M_1}\in X$ such that $A\subset \{x_1, x_2,\ldots, x_{M_1}\}^{M_2, \gamma}$.
\end{enumerate}

\end{theorem}

\begin{example}\label{ejBb} \rm
Consider $B$, the closed unit ball of $l_2$ and $\{e_n:n\in\N\}$ its usual basis. 
Let $X\subset l_2$ be given by $X=B\bigcup\{x_n^m:n\in\N, m=1,\ldots, n\}$, 
where $x_n^m=\left(1+\frac mn\right)e_n$. Then, $X$ is Bourbaki-bounded. 
\end{example}

\begin{proof}
We will show that $X$ is finitely chainable. Let $\gamma>0$ and $k\in\N$ such that 
$\frac 1k \leq \gamma<\frac 1{k-1}$. Then, there are finitely many points in $X$ such that 
$I(x)>\frac 1k$ --namely, $\{x_n^m: n<k, m=1,\ldots, n\}$. So we just need 
to {\em chain} $X_k=\{x\in X:I(x)\leq\frac 1k\}$. 
Now, beginning at the origin 0, we have: 

$A_0=0, A^{1,1/k}=B\left[0,\frac 1k\right], 
A^{2,1/k}=B\left[0, \frac 2k\right], \ldots, A^{k, 1/k}=B\left[0,1\right],
A^{k+1,1/k}=B\cup \left\{x_n^m:d(x_n^m, B)\leq \frac 1k\right\}, $

$A^{k+2, 1/k}=\left\{x_n^m:d(x_n^m, A^{k+1,1/k})\leq \frac 1k\right\}, \ldots, 
A^{3k}=X_k.$

Here, the worst subset to chain is $\{x^m_{2k-1}:m= 1, \ldots, 2k-1\}$, since every 
enlargement gives us just one more point than what we had. So, after $A_k$, we 
still need $2k-1$ enlargements to cover the whole space. In any case, we do not 
need more than $3k$ steps to get to any point in $X_k$ from the origin. 
As $X\setminus X_k$ contains finitely many points, $X$ is Bourbaki-bounded. 
\end{proof}

\begin{remark} \rm
Please note that whenever $X$ is a Bourbaki-bounded space, it is 
Bourbaki-bounded when considered as a subset of another metric space, but 
not every Bourbaki-bounded subset $A\subset X$ is a Bourbaki-bounded space. 
\end{remark}

\begin{example} \rm
Let, again, $\{e_n:n\in\N\}$ be the usual basis of $l_2$. As shown in the 
previous example, it is a Bourbaki-bounded subset of $X$ (every subset is), 
but it is not a Bourbaki-bounded space. 
\end{example}

\section{The main result} 

It is time to state everything properly. 

\begin{nota}
For the sake of clarity, we must explicitly recall the notion of Atsuji 
isolation index: $I(x)=d(x,X\setminus \{x\})=\inf\{d(x,y):y \in X, y\neq x\}$ 
for every $x\in X$. 
\end{nota}

\begin{definition} \rm
$A\subset X$ is uniformly isolated if 
$\inf\{I(a):a\in A\}>0$. This is equivalent to the existence 
of $\e>0$ such that $d(a,x)\geq\e$ for every $a\in A, x\in X$. 
\end{definition}

\begin{lemma}\cite{McS}
For any $A\subset X$ and any $f_0\in U^*(A)$, there exists $f\in U^*(X)$ 
such that $f_{|A}=f_0$. 
\end{lemma}

\begin{remark} \rm
For any couple of sequences $(x_n), (y_n)\subset X$ such that 
$d(x_n, x_m)\geq \e>0, 0<d(y_n, x_n)\leq\min\{\e/3, 1/n\}$ for every 
$n\neq m\in\N$ and any $\alpha_n\to 0$, the function 

$g_0:\{x_n:n\in\N\}\cup\{y_n:n\in\N\}\to \R,\ g_0(x_n)=\alpha_n, g_0(y_n)=0$

\noindent is uniformly continuous. As it is bounded, too, the previous lemma shows that 
we can extend $g_0$ to $g\in U^*(X)$. This extension will be useful in the proof 
of the following result. 
\end{remark}

\begin{theorem}
Let $(X,d)$ be a metric space. Then, $U(X)$ is a ring if and only if every 
non Bourbaki-bounded $A\subset X$ contains an infinite uniformly isolated subset. 
\end{theorem}

\begin{proof}
{\bf The ``only if" implication:} 

\vspace{1 mm}

Suppose, on the contrary, that there exists a non Bourbaki-bounded $A\subset X$ 
such that for every $\delta>0$, $\{x\in A:I(x)\geq\delta\}$ is finite. Now, take 
$f\in U(X)$ unbounded on $A$ and $(x_n)\subset X$ such that $f(x_n)\geq f(x_{n-1})+1\geq n$, for 
every $n\in\N$. As $|f(x_n)-f(x_m)|\geq 1$ for every $m\neq n\in\N$ and $f$ is 
uniformly continuous, there exists $\e$ such that $d(x_n,x_m)\geq\e$ when $m\neq n$. 
As $I(x_n)$ tends to 0, we may take another sequence $(y_n)\subset X$ such that 
$\frac \e 3\geq d(x_n, y_n)\to 0$. The function 
$h_0:\{x_n:n\in\N\}\cup\{y_n:n\in\N\}\to [0,1],$ defined by $h_0(x_n)=\frac 1n, h_0(y_n)=0$ 
is uniformly continuous and bounded. So, we may extend $h_0$ to $h\in U^*(X)$ and 
we have two options: 

\begin{itemize}
\item $f(x_n)\leq f(y_n)$ for infinitely many $n$. 
\item $f(x_n)\geq f(y_n)$ for infinitely many $n$. 
\end{itemize}

In the first case, consider $f-h\in U(X)$. A simple calculation shows that its 
square cannot be uniformly continuous: 

\vspace{2 mm}

$$\left((f-h)^2\right)(x_n)-\left((f-h)^2\right)(y_n)=$$
$$=f^2(x_n)+h^2(x_n)-2f(x_n)h(x_n)-\left[f^2(y_n)+h^2(y_n)-2f(y_n)h(y_n)\right]=$$
$$=f^2(x_n)+ \frac 1{n^2}-2\frac 1nf(x_n) -f^2(y_n)\leq \frac 1{n^2}-2\leq -1$$ 
for infinitely many $x_n, y_n$ such that $d(x_n, y_n)\leq \frac 1n$, so 
$(f-h)^2$ is not uniformly continuous. In the second case, it is enough to 
consider $f+h$ instead of $f-h$ and we have showed that $U(X)$ is not a ring. 
\vspace{3 mm}

{\bf The ``if" implication:} 

\vspace{1 mm}

Suppose $U(X)$ is not a ring. Then, there exist $f, g\in U(X)$ such that 
$f\cdot g\not\in U(X)$, so there are $\e>0$ and sequences $(x_n), (y_n) 
\subset X$ such that $d(x_n, y_n)\leq \frac 1n$ and 
$|(f\cdot g)(x_n) - (f\cdot g)(y_n)|\geq\e$. As $f\cdot g$ is uniformly 
continuous whenever $f, g\in U^*(X)$, this implies that either $f$ or $g$ is 
unbounded on $A=\{x_n:n\in\N\}\cup\{y_n:n\in\N\}$. As both $I(x_n), I(y_n)$ are 
not greater than $\frac 1n$ because $d(x_n, y_n)\leq \frac 1n$, $A$ is a non 
Bourbaki-bounded subset such that $A\cap\{x:I(x)>\delta\}$ is finite for 
every $\delta>0$. 
\end{proof}

\begin{corollary}
Suppose there exists a Bourbaki-bounded subset $F\subset X$ such that 
$X\setminus (F^\gamma)$ is uniformly isolated for every $\gamma>0$. Then, 
$U(X)$ is a ring. 
\end{corollary}

\begin{proof}
It can be easily deduced from the above theorem, but we will give an alternative 
proof: 
we will show directly that $f\cdot g$ is uniformly continuous. 

Suppose $F$ is Bourbaki-bounded and such that $\{x\in X:d(x,F)\geq\gamma\}$
is uniformly isolated fo every $\gamma$ and let $f, g\in U(X)$. Then, since
$F$ is Bourbaki-bounded, there exist $M, N\in \N$ such that $f_{|F}\leq M$
and $g_{|F}\leq N$. As both functions are uniformly continuous, there exist
$\gamma_f, \gamma_g$ such that $d(x,y)<\gamma_f$ implies $|f(x)-f(y)|<1$ and
$d(x,y)<\gamma_g$ implies $|g(x)-g(y)|<1$. So, if we take
$\gamma=\min\{\gamma_f, \gamma_g\}$, we have $\sup\{f(x):x\in F^\gamma\}\leq M+1$
and $\sup\{g(x):x\in F^\gamma\}\leq N+1$.

Now, there exists $\alpha>0$ such that, $x\not\in F^{\gamma}$ implies 
$I(x)\geq\alpha$ so, whenever $d(x,y)<\alpha$, both $x$ and $y$ must belong to $F^{\gamma}$.

We need to show that for every $\e>0$ exists $\delta$ such that $d(x,y)<\delta$
implies $|g(x)f(x)-g(y)f(y)|<\e$. So let $\e>0$ and take $\delta_f, \delta_g$
such that $d(x,y)<\delta_f$ implies $|f(x)-f(y)|<\frac\e{2(M+1)}$ and
$d(x,y)<\delta_g$ implies $|g(x)-g(y)|<\frac\e{2(N+1)}$. Now, taking
$\delta=\min\left\{\alpha, \delta_f, \delta_g\right\}$, we obtain,
for $x, y$ such that $d(x,y)<\delta$:

\vspace{3 mm}
$|f(x)g(x)-f(y)g(y)|=|f(x)g(x)-f(y)g(x)+f(y)g(x)-f(y)g(y)|\leq
|f(x)g(x)-f(y)g(x)|+|f(y)g(x)-f(y)g(y)|=$

\vspace{3 mm}

$=|g(x)(f(x)-f(y))|+ |f(y)(g(x)-g(y))|\leq (M+1)\frac\e{2(M+1)} + (N+1)\frac\e{2(N+1)}=\e,$
and so, $fg\in U(X)$.
\end{proof}

\begin{question}
Is this an equivalence? 
\end{question}

\section*{Acknowlegments}
The author feels deeply grateful to Prof.~Maribel Garrido, for several useful 
discussions, references and advices on this topic.


\begin{thebibliography}{33}

\bibitem{A}
M. Atsuji, Uniform continuity of continuous functions of metric spaces. 
Pac. J. Math. 8 (1958), 11--16.

\bibitem{BN}
G. Beer, S. Naimpally, Uniform continuity of a product of real functions. 
Real Anal. Exch. 37, no. 1 (2011/2012) 213--220. 

\bibitem{C}
J. Cabello, A sharp representation of multiplicative isomorphisms of 
uniformly continuous functions, submitted. 


\bibitem{GM}
M.I. Garrido, A.S. Mero\~no, New types of completeness in metric spaces. 
Annales Academi\ae\ Scientiarum Fennic\ae\ 39, no. 2 (2014) 733--758. 

\bibitem{H}
J. Hejcman, Boundedness in uniform spaces and topological groups.  
Czechoslovak Math.J. 9, 1959, 544--563.

\bibitem{McS}
E.J. McShane, Extension of range of functions.
Bull. Amer. Math. Soc. 40, 1934, 837--842. 

\bibitem{N}
S.B. Nadler, Pointwise products of real uniformly continuous functions. 
Sarajevo J. Math 1, no. 1 (2005) 117--127. 

\bibitem{NZ}
S.B. Nadler, D.M. Zitney, Pointwise products of uniformly continuous 
functions on sets in the real line. 
American Mathematical Monthly 114, no. 2 (2007) 160--163. 

\bibitem{Nj}
O. Nj\aa stad, On uniform spaces where all uniformly continuous functions are bounded. 
Montsh. Math. 69 (1965) 167-176. 

\end{thebibliography}
\end{document}